\documentclass[12pt]{article}

\usepackage{graphicx}
\usepackage{wrapfig}
\usepackage{amssymb,amsfonts,amsmath,amsthm,amscd}
\usepackage{color}
\usepackage[english]{babel}
\usepackage[latin1]{inputenc}
\usepackage{times}
\usepackage{geometry}
\usepackage{newtxtext,newtxmath}
\usepackage{pgf,tikz,circuitikz}\usetikzlibrary{arrows}
\usepackage{mathrsfs}

\newtheorem{theorem}{Theorem}[section]
\newtheorem{proposition}[theorem]{Proposition}
\newtheorem{lemma}[theorem]{Lemma}

\newtheorem{corollary}[theorem]{Corollary}
\newtheorem*{theorem*}{Theorem}

\theoremstyle{remark}

\newtheorem*{remark*}{Remark}

\theoremstyle{definition}

\newcommand{\ploop}{P_{\!\!\circ}}

\begin{document}

\title{A generalization of Cardy's and Schramm's formulae}

\author{Mikhail Khristoforov, Mikhail Skopenkov, Stanislav Smirnov}

\date{}

\maketitle

\begin{abstract} We study critical site percolation on the triangular lattice. We find the difference of the probabilities of having a percolation interface to the right and to the left of two given points in the scaling limit. This generalizes both Cardy's and Schramm's formulae. The generalization involves a new interesting discrete analytic observable and an unexpected conformal mapping.

\smallskip


\end{abstract}

\begin{figure}[!h]
\begin{center}
  \includegraphics[width=0.22\textwidth]{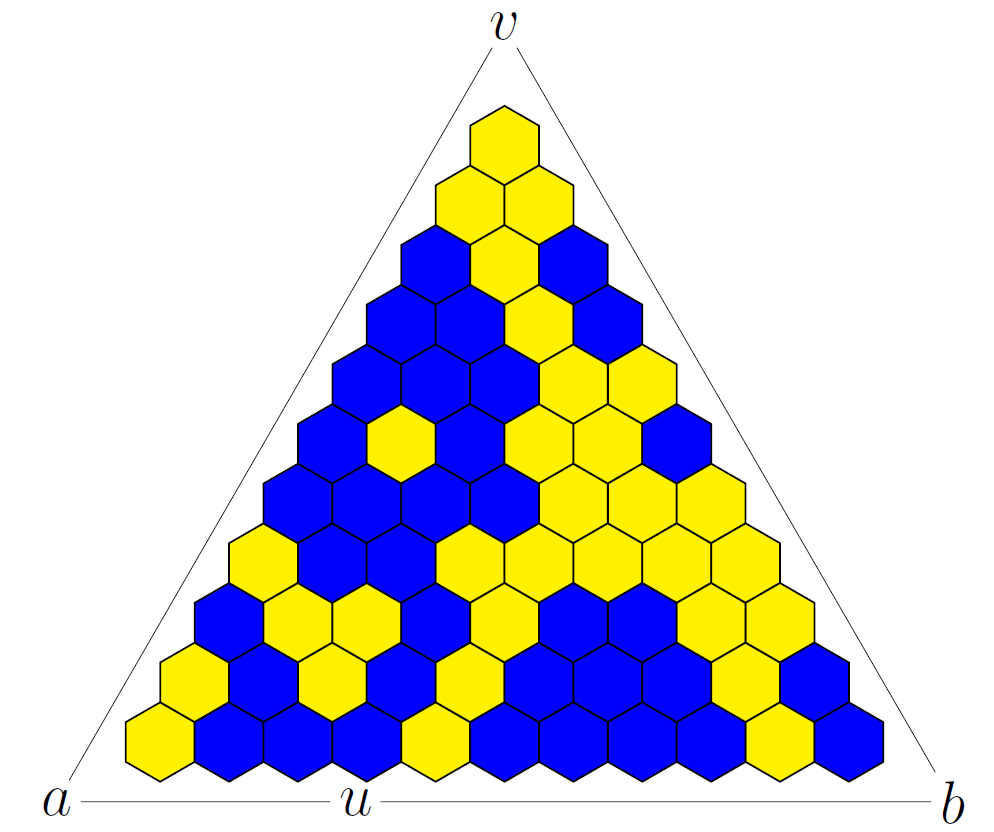}
\end{center}
  \caption{The percolation model; see Corollary~\ref{cor-Cardy}}
  \label{fig-simulation}
\end{figure}

\section{Introduction}

Percolation is an archetypical model of phase transition, used to describe many natural phenomena, from the spread of epidemics to a liquid seeping through a porous medium. It was introduced by Broadbent and Hammersley \cite{Broadbent-Hammersley-57} but appeared even earlier in the problem section of the American Mathematical Monthly \cite{Wood-94}.

In the simplest setup of \emph{Bernoulli percolation}, vertices (\emph{sites}) or edges (\emph{bonds}) of a graph are independently declared open or closed with probabilities $p$ and $1-p$ correspondingly; the resulting model is called \emph{site} or \emph{bond percolation} correspondingly. Connected clusters of open sites (or bonds) are then studied. For example, one can ask how the probability of having an open cluster connecting two sets depends on $p$. Despite simple formulation, the model approximates physical phenomena quite well and exhibits a very complicated behavior.

There is an extensive theory, see e.g.~\cite{Grimmett-99}, but still there are many open questions. For instance, it is not known whether probability $\theta(p)$ of having an infinite open cluster containing origin in $\mathbb{Z}^3$ depends continuously on $p$.

In contrast, planar models are fairly well understood. We will study critical site percolation on the triangular lattice, which by duality can be represented as a random coloring of faces (plaquettes) on the hexagonal lattice; see Figure~\ref{fig-simulation}, where open and closed hexagons are represented by blue and yellow colors. For this model, it is known that the critical value $p_c$ is equal to $1/2$, which was first proved by Kesten using duality; see~\cite{Grimmett-99}.
This means in particular that when one takes a topological rectangle (i.e. a domain with four marked boundary points) and superimposes a mesh $\delta$ lattice, the crossing probability (i.e. the probability of the existence of an open crossing between the chosen opposite sides) tends to $0$ when $p<p_c$ and to $1$ when $p>p_c$ as $
\delta\searrow 0$. The same argument shows that for $p=p_c$ the crossing probability is nontrivial, and in 1992 physicist J.~Cardy suggested an exact formula for its scaling limit as $\delta\searrow 0$ as a hypergeometric function of the conformal modulus (see Corollary~\ref{cor-Cardy}). The formula is expected to hold for any critical percolation, but so far it was proved only for the model under consideration by the third author \cite{Smirnov-01, Smirnov-09, KS-20} by establishing discrete holomorphicity of certain observables.

Henceforth one can connect \cite{Smirnov-03} scaling limits of percolation interfaces between open and closed clusters to Schramm's SLE(6) curve \cite{Schramm-00, Schramm-01} and deduce many properties, e.g. the values of critical exponents and dimensions \cite{Smirnov-Werner-01}.

This leaves open the question of how far one can go with purely discrete techniques. For example, in the critical Ising model much can be learned from similar observables without alluding to SLE \cite{Hongler-Smirnov-13, Chelkak-Hongler-Izyurov-15}, as is the case for the Uniform Spanning Tree \cite[\S11.2]{Kenyon-11}. Cf.~an elementary example in~\cite{Novikov-19}.

In this paper, we introduce a new percolation observable, which allows us to prove new results and give new proofs to old ones, which beforehand required an SLE limit (see Corollary~\ref{cor-Schramm}). In particular, we find the difference of the probabilities of having a percolation interface to the right and to the left of \emph{two} given points in the scaling limit (see Figure~\ref{fig1} and Theorem~\ref{th-main}).

Remarkably, the observable we introduce here is nothing more than a natural generalization
of the observable from \cite{KS-20} which itself is just equal to the observable from \cite{Smirnov-01}.
However, it seems that getting to the new observable directly from \cite{Smirnov-01} was not as straightforward (and so took a while).
It is worth noting that our construction admits further generalizations, \emph{ at least} to the problem with six \emph {disorders} instead of four \cite{KK-22} and that some appearing functions bear resemblance to modular forms
discussed by P.~Kleban and D.~Zagier \cite{Kleban-Zagier-03}.

\textbf{Organization of the paper.}
In \S\ref{sec-statements} and \S\ref{sec-preliminaries} we state main and auxiliary results respectively. In \S\ref{sec-proofs} and Appendix~\ref{sec-technical} we prove conceptual and technical ones respectively.

\begin{figure}[htbp]
\begin{center}
\begin{tabular}{cc}
\includegraphics[width=3.8cm]{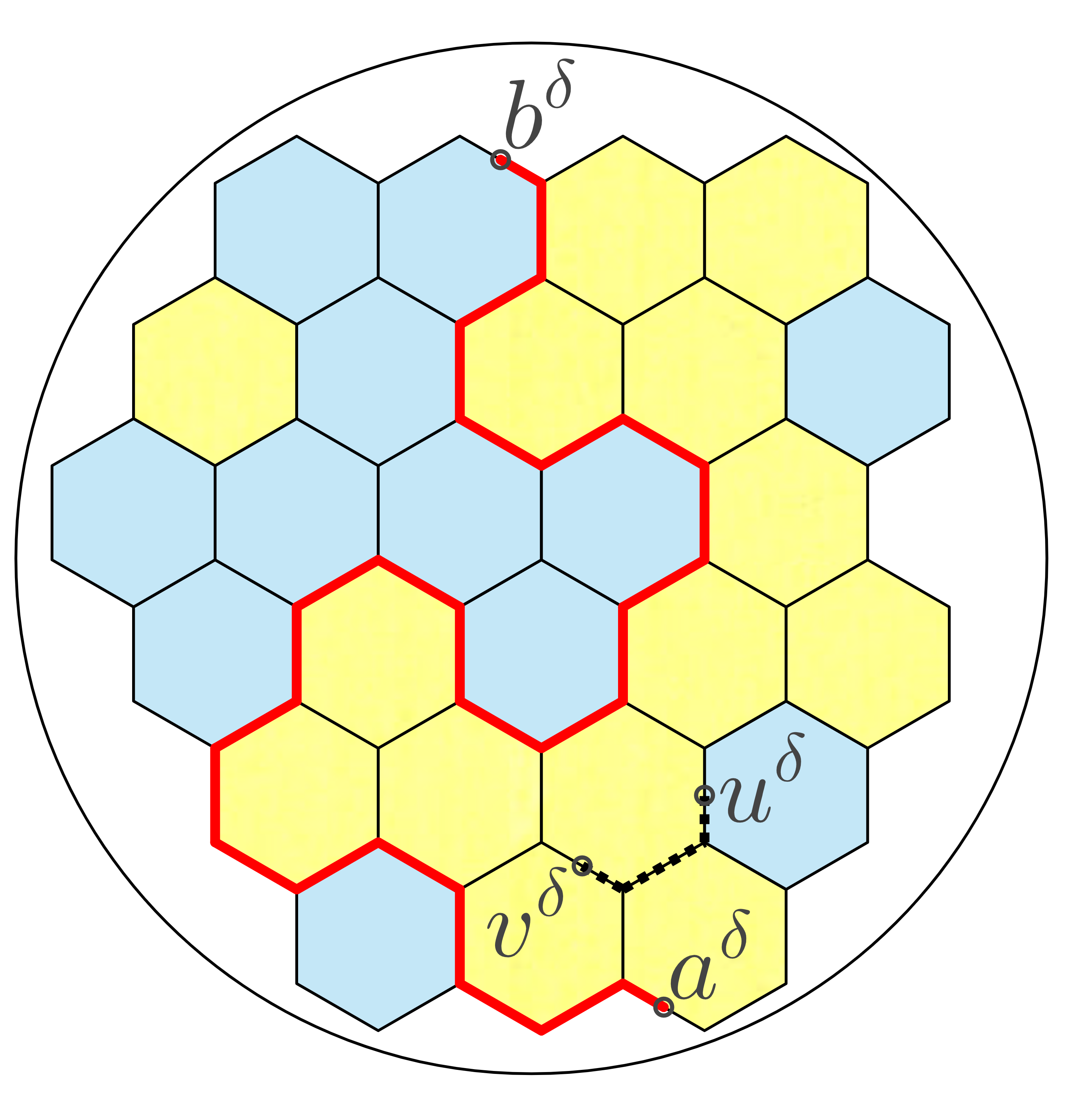} &
\includegraphics[width=3.8cm]{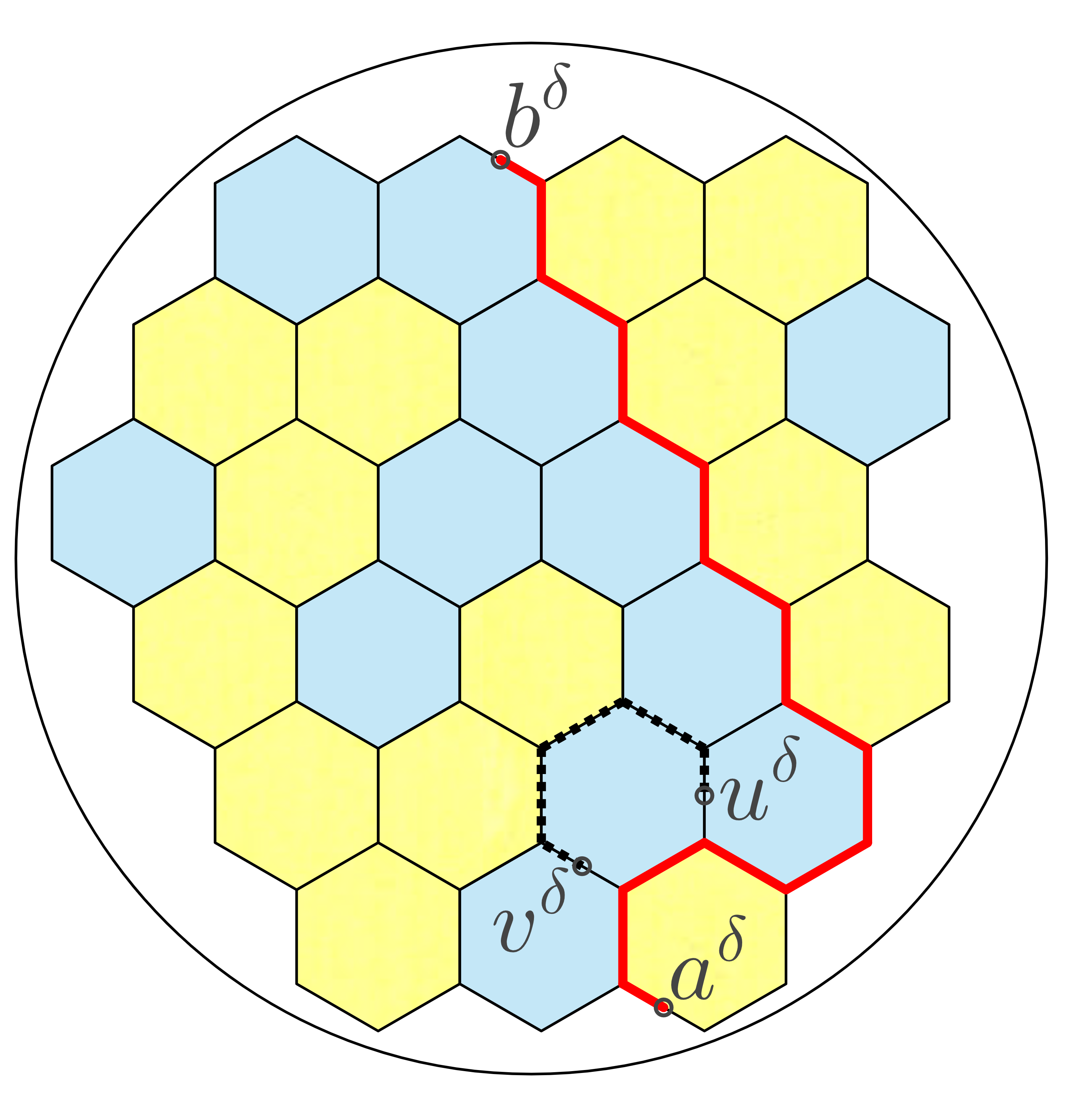}
\\
$P(a^\delta \rightarrow b^\delta,u^\delta \cdots v^\delta)$  &
$P(u^\delta\cdots v^\delta,a^\delta \rightarrow b^\delta)$
\end{tabular}
\end{center}
\caption{Colorings with the interfaces $a^\delta b^\delta$ passing to the left and to the right from given points $u^\delta,v^\delta$,
and the probabilities those colorings contribute to.
The dashed paths demonstrate that $u^\delta$ and $v^\delta$ are in the same connected component of the union of black sides, which is one of our requirements.}
\label{fig1}
\end{figure}

\section{Statement} \label{sec-statements}

Let us introduce a few definitions to state our result precisely  (see Figure~\ref{fig1}).

Consider a hexagonal lattice on the complex plane $\mathbb C$ formed by hexagons with side $\delta$. Let $\Omega\subset \mathbb C$ be a domain bounded by a closed smooth curve (that is, $\partial\Omega$ is the image of a periodic $C^1$ map $\mathbb{R}\to\mathbb{C}$ with nonzero derivative).
A \emph{lattice approximation} of the domain $\Omega$ is the maximal-area connected component  $\Omega^\delta$ of the union of all the hexagons lying inside $\Omega$ (if such a connected component is not unique, then we choose any of them). Mark two distinct boundary points $a,b\in \partial \Omega$ and two other points $u,v\in \overline{\Omega}:=\Omega \cup \partial\Omega$. Their \emph{lattice approximations} are the midpoints $a^\delta,b^\delta,u^\delta,v^\delta$ of sides of the hexagons of $\Omega^\delta$, closest to $a,b,u,v$ respectively (if the closest midpoint is not unique, then we choose any of them). Clearly, $a^\delta,b^\delta\in \partial\Omega^\delta$.
We allow $u^\delta=v^\delta$ but disallow the coincidence of any other pair among $a^\delta,b^\delta,u^\delta,v^\delta$.

The \emph{percolation model} on $\Omega^\delta$ is the uniform measure on the set of all the colorings of hexagons of $\Omega^\delta$ in two colors, say, blue and yellow.
Introduce the \emph{Dobrushin boundary condition}: in addition, paint the hexagons \emph{outside} $\Omega^\delta$ bordering upon the clockwise arc of $\partial\Omega^\delta$ between $a^\delta$ and $b^\delta$ blue, and the ones bordering upon the counterclockwise arc yellow (the ones bordering upon $a^\delta$ and $b^\delta$ are paint both colors). For the whole coloring, the \emph{interface} $a^\delta b^\delta$ is the oriented simple broken line going from $a^\delta$ to $b^\delta$ along the sides of the hexagons of $\Omega^\delta$ such that all the hexagons bordering upon $a^\delta b^\delta$ from the left are blue and all the hexagons bordering upon $a^\delta b^\delta$ from the right are yellow.

The interface $a^\delta b^\delta$ splits the union of the sides of the hexagons of $\Omega^\delta$ into connected components. Let $P(u^\delta\cdots v^\delta,a^\delta \rightarrow b^\delta)$ (respectively, $P(a^\delta \rightarrow b^\delta,u^\delta \cdots v^\delta)$) be the probability that $u^\delta$ and $v^\delta$ belong to the same component lying to the left (respectively, to the right) from $a^\delta b^\delta$.

Our main result expresses those probabilities in terms of certain conformal mappings. Let $\psi$ be a conformal mapping of $\Omega$ onto the upper half-plane $\{z:\mathrm{Im}\,z>0\}$ (continuously extended to $\overline{\Omega}$) taking $a$ and $b$ to $0$ and $\infty$ respectively. Let $g$ be the conformal mapping of $\mathbb{C}-(-\infty;-1]\cup[1;+\infty)$ onto the interior of the rhombus with the vertices $\pm 1,\pm i\sqrt{3}$,  continuously extended to the points $\pm 1$ and having fixed points $0,+1,-1$. See Figure~\ref{fig:g}.

\begin{figure}[htbp]
  \centering
      \includegraphics[height=3.6cm]{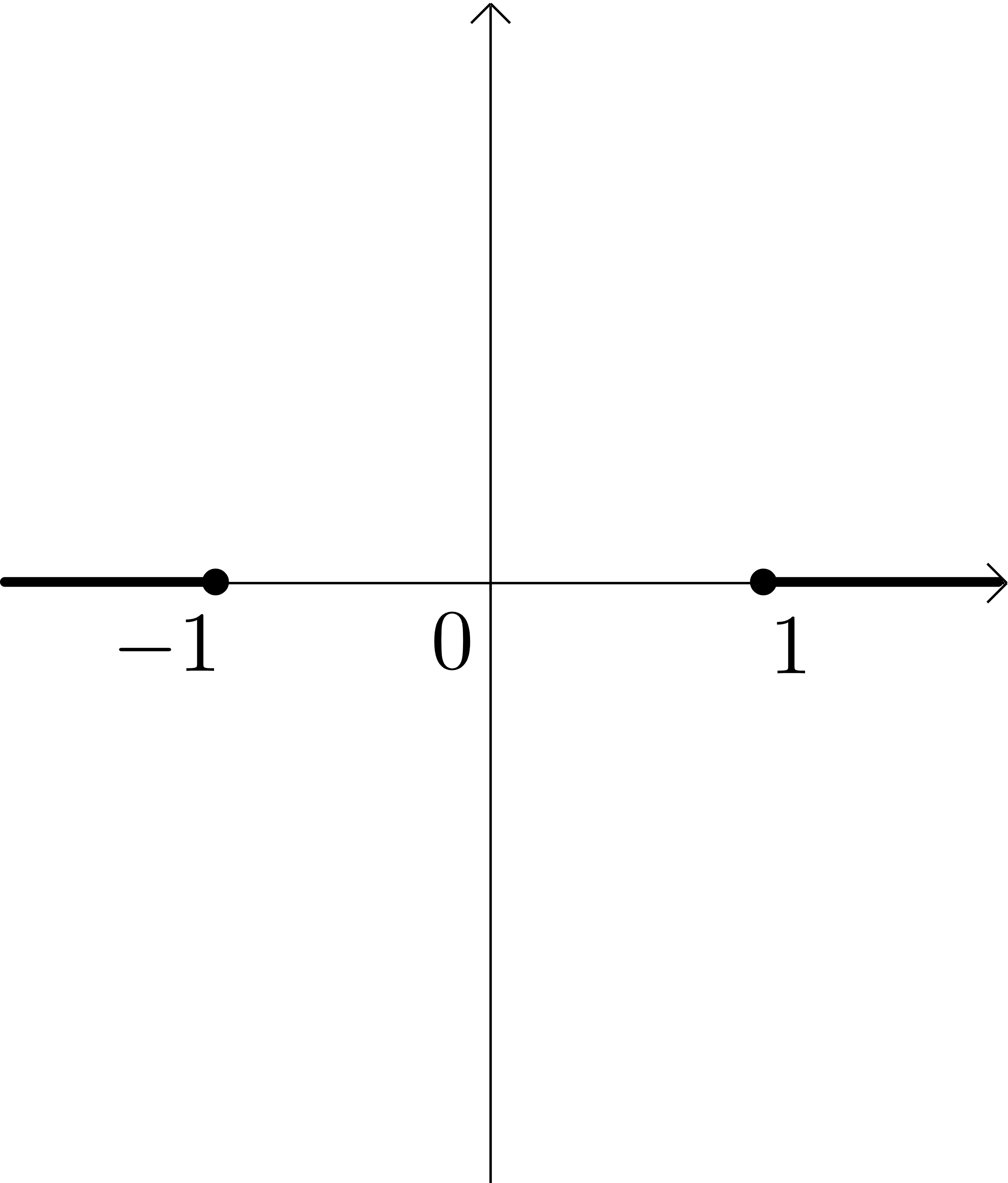}
      \includegraphics[height=3.6cm]{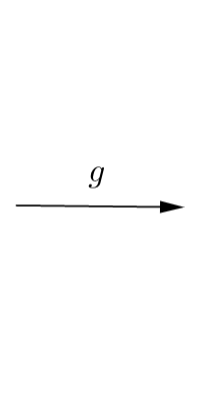}
      \includegraphics[height=3.6cm]{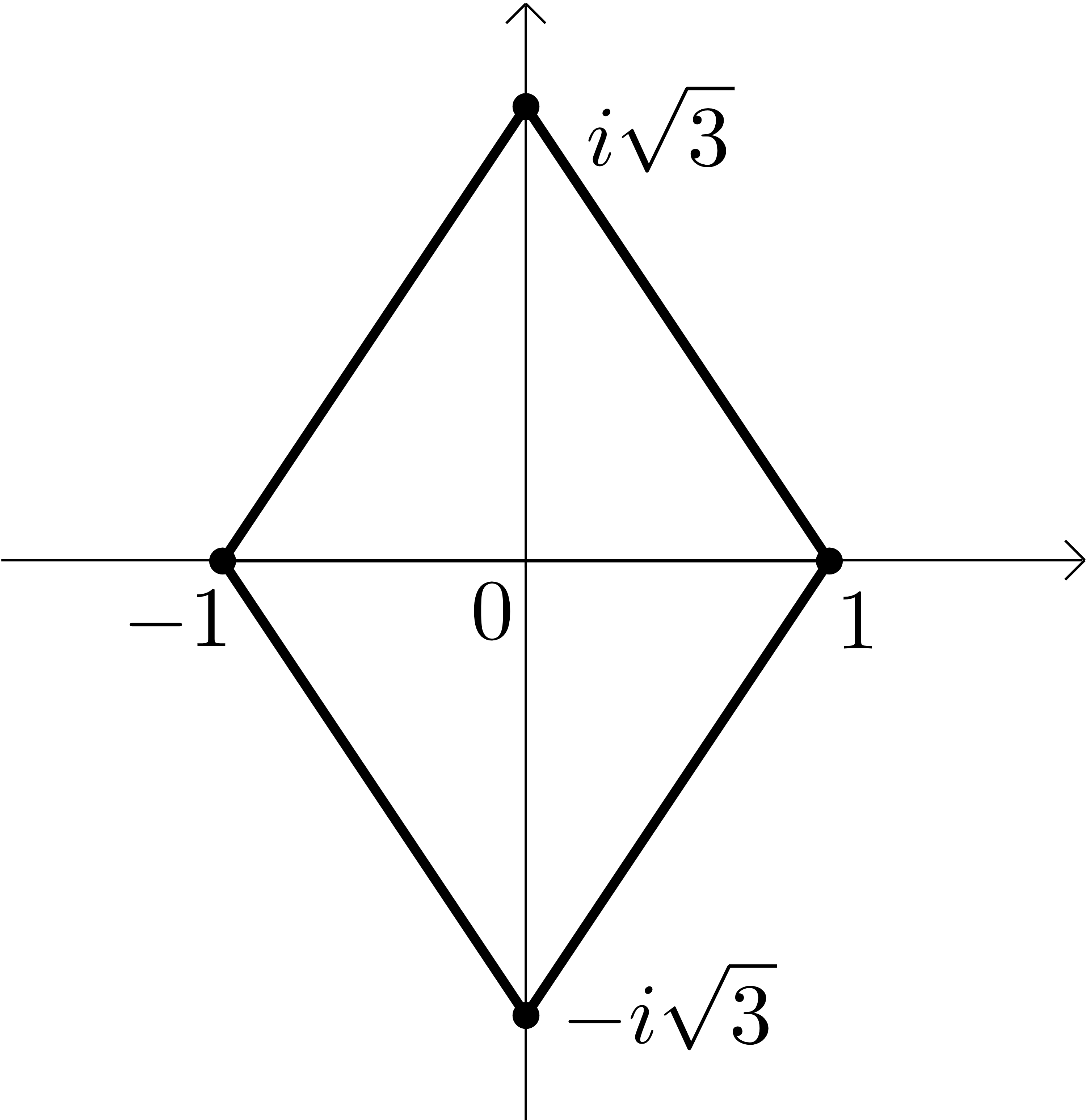}
  \caption{The conformal mapping $g(z)$ of the plane pierced by two slits onto a lozenge}\label{fig:g}
\end{figure}

\begin{theorem}\label{th-main} Let $(\Omega^\delta,a^\delta,b^\delta,u^\delta,v^\delta)$ be a lattice approximation of a domain $(\Omega,a,b,u,v)$ with two marked distinct boundary points $a,b\in\partial \Omega$ and two other points $u,v\in \overline{\Omega}$ such that $\partial \Omega$ is a closed smooth curve. Then
\begin{equation}\label{eq-th-main}
P(u^\delta \cdots v^\delta,a^\delta \rightarrow b^\delta)-
P(a^\delta \rightarrow b^\delta,u^\delta\cdots v^\delta)\to
\frac{1}{\sqrt{3}}\mathrm{Im}\,g\left(
\frac{\psi(u)+\overline{\psi(v)}}{\psi(u)-\overline{\psi(v)}}\right)
\qquad\text{as }\delta\searrow 0.
\end{equation}
\end{theorem}

Here if both $u$ and $v$ belong to the counterclockwise (respectively, clockwise) boundary arc $ab$ of $\partial \Omega$, then the value of the mapping $g$ in~\eqref{eq-th-main} is understood as continuously extended from the lower (respectively, upper) half-plane. There is an explicit formula for the mapping $g$.

\begin{proposition}\label{l-schwarz}
We have
\begin{equation}\label{eq-l-schwarz}
g(z)=
\frac {6\,\Gamma(2/3)}{\Gamma(1/3)^2}
\left(\frac {z+1}2 \right)^{1/3}\cdot {}_2F_1 \left(\frac 13, \frac 23; \frac 43; \frac {z+1}2  \right)-1
=\frac{2\sqrt{3}\,\Gamma(2/3)}{\sqrt{\pi}\,\Gamma(1/6)}
\cdot z\cdot
{}_2F_{1}\left(\frac{1}{2},\frac{2}{3};\frac{3}{2};
z^2\right).
\end{equation}
\end{proposition}

Hereafter
$$
{}_2F_{1}(p,q;r;z):=\frac {\Gamma(r)}{\Gamma(q)\Gamma(r-q)}
\int_{0}^{1}t^{q-1}(1-t)^{r-q-1}(1-tz)^{-p}\,dt
$$
denotes the \emph{principal branch of the hypergeometric function} in $\mathbb{C}-[1,+\infty)$; \cite[Ch.~V, \S7]{Nehari-75}.

\begin{remark*} The real part of $g(z)$ has a probabilistic meaning as well; see Theorem~\ref{prop-convergence} below.
\end{remark*}

\begin{remark*} The smoothness of $\partial\Omega$ is not really a restriction. Using the methods of \cite{KS-20}, one can generalize this result to a Jordan domain~$\Omega$ (and even an arbitrary bounded simply-connected domain~$\Omega$, if $\partial \Omega$ is understood as the set of prime ends of $\Omega$ and $(\Omega^\delta, a^\delta,b^\delta,u^\delta,v^\delta)$ converges to $(\Omega,a,b,u,v)$ in the Caratheodory sense). However, the message of the paper can be seen already in the simplest particular case $\Omega=\mathrm{Int}\, D^2:=\{z\in\mathbb{C}:|z|<1\}$.
\end{remark*}

\begin{remark*}
For $u,v\in \partial \Omega$ the theorem gives Cardy's formula for the crossing probability, and for $u=v$ it gives Schramm's formula for the surrounding probability.
The \emph{crossing probability} $P(a^\delta u^\delta\leftrightarrow b^\delta v^\delta)$ is the probability that some connected component of the union of blue hexagons of $\Omega^\delta$ has common points with both arcs $a^\delta u^\delta$ and $b^\delta v^\delta$ of $\partial\Omega^\delta$ (which means just a blue crossing between the arcs). The \emph{surrounding} (or \emph{right-passage}) \emph{probability} $P(v,a^\delta \rightarrow b^\delta)$ is the probability that $v$ belongs to a connected component of the complement $\Omega^\delta-a^\delta b^\delta$ bordering upon the interface $a^\delta b^\delta$ from the left (which means surrounding of $v$ by the interface $a^\delta b^\delta$ and the clockwise boundary arc $b^\delta a^\delta\subset \partial\Omega^\delta$).
\end{remark*}

\begin{corollary}[Cardy's formula] \label{cor-Cardy}
\textup{\cite[Eq.~(8)]{Cardy-92}}
Let $(\Omega^\delta,a^\delta,b^\delta,v^\delta,u^\delta)$ be a lattice approximation of a domain $(\Omega,a,b,v,u)$ bounded by a closed smooth curve with four distinct marked points lying on the boundary $\partial \Omega$ in the counterclockwise order
$a,b,v,u$. Then
\begin{equation}\label{eq-cor-Cardy}
P(a^\delta u^\delta\leftrightarrow b^\delta v^\delta)
  \to g_\Omega(u)=\frac {3\Gamma(2/3)}{\Gamma(1/3)^2} \left(\frac{\psi(u)}{\psi(v)}\right)^{1/3}\cdot
  {}_2F_1 \left(\frac 13, \frac 23; \frac 43; \frac{\psi(u)}{\psi(v)} \right)
  \qquad\text{as }\delta\searrow 0,
\end{equation}
where $g_\Omega$ is the conformal mapping of $\Omega$ onto the equilateral triangle with the vertices $0,1,(1+\sqrt{3}i)/2$ (continuously extended to $\overline\Omega$) taking $a,b,v$ to the respective vertices.
\end{corollary}

\begin{corollary}[Schramm's formula] \label{cor-Schramm}
\textup{\cite[Theorem~1]{Schramm-01}}
Let $(\Omega^\delta,a^\delta,b^\delta)$ be a lattice approximation of the unit disk $(\mathrm{Int}\,D^2,a,b)$ with two distinct marked boundary points $a,b\in\partial D^2$. Then
\begin{equation}\label{eq-cor-Schramm}
  P(0,a^\delta \rightarrow b^\delta)\to
  \frac{1}{2}-\frac{\Gamma(2/3)}{\sqrt{\pi}\Gamma(1/6)}
  \cot\frac{\theta}{2}\cdot
  {}_2F_{1}\left(
  \frac{1}{2},\frac{2}{3};\frac{3}{2};-\cot^2\frac{\theta}{2}\right)
  \text{ as }\delta\searrow 0,\text{ where } e^{i\theta}:=\frac{a}{b}.
\end{equation}
\end{corollary}

While the proof of Theorem~\ref{th-main} is similar to the proof of Cardy's formula from \cite{KS-20}, the proof of Schramm's formula by means of Theorem~\ref{th-main} is essentially new, and {we think} is simpler than the original one. Originally, Schramm's formula was stated as a corollary of the weak convergence of interfaces to SLE(6), itself deduced from Cardy's formula~\cite{Smirnov-01}; but actually, it is rather technical to deduce  the convergence of surrounding probabilities from the latter \emph{weak} convergence.

\begin{remark*}
In this text we refrain from referring to the spinor percolation point of view \cite[\S1.3]{KS-20}, however
a reader might find it useful to think in those terms. In particular it would make the definition of $\ploop(\dots)$ below,
Lemma \ref{l-symmetric-difference}, and the proof of Lemma \ref{l-uniform-holderness} more transparent.
\end{remark*}

\section{Preliminaries}\label{sec-preliminaries}

Theorem~\ref{th-main} is deduced from a more general result, interesting in itself.
Let us introduce some notation, state the main result in full strength, and state lemmas used in the proof, giving proofs in the next sections.
Throughout this section we assume that the assumptions of Theorem~\ref{th-main} hold, we omit the superscript $\delta$ so that $a,b,u,v$ mean $a^\delta,b^\delta,u^\delta,v^\delta$ respectively (everywhere except Theorem~\ref{prop-convergence}, Lemmas~\ref{l-lattice-approximation}--\ref{l-mapping}, and other explicitly indicated places), and assume that $a,b,u,v$ are pairwise distinct. Remarkably, although we take $u=v$ in the proof of Schramm's formula at the very end, here we need~$u\ne v$.

\begin{figure}[htbp]
  \centering
  \includegraphics[width=0.14\textwidth]{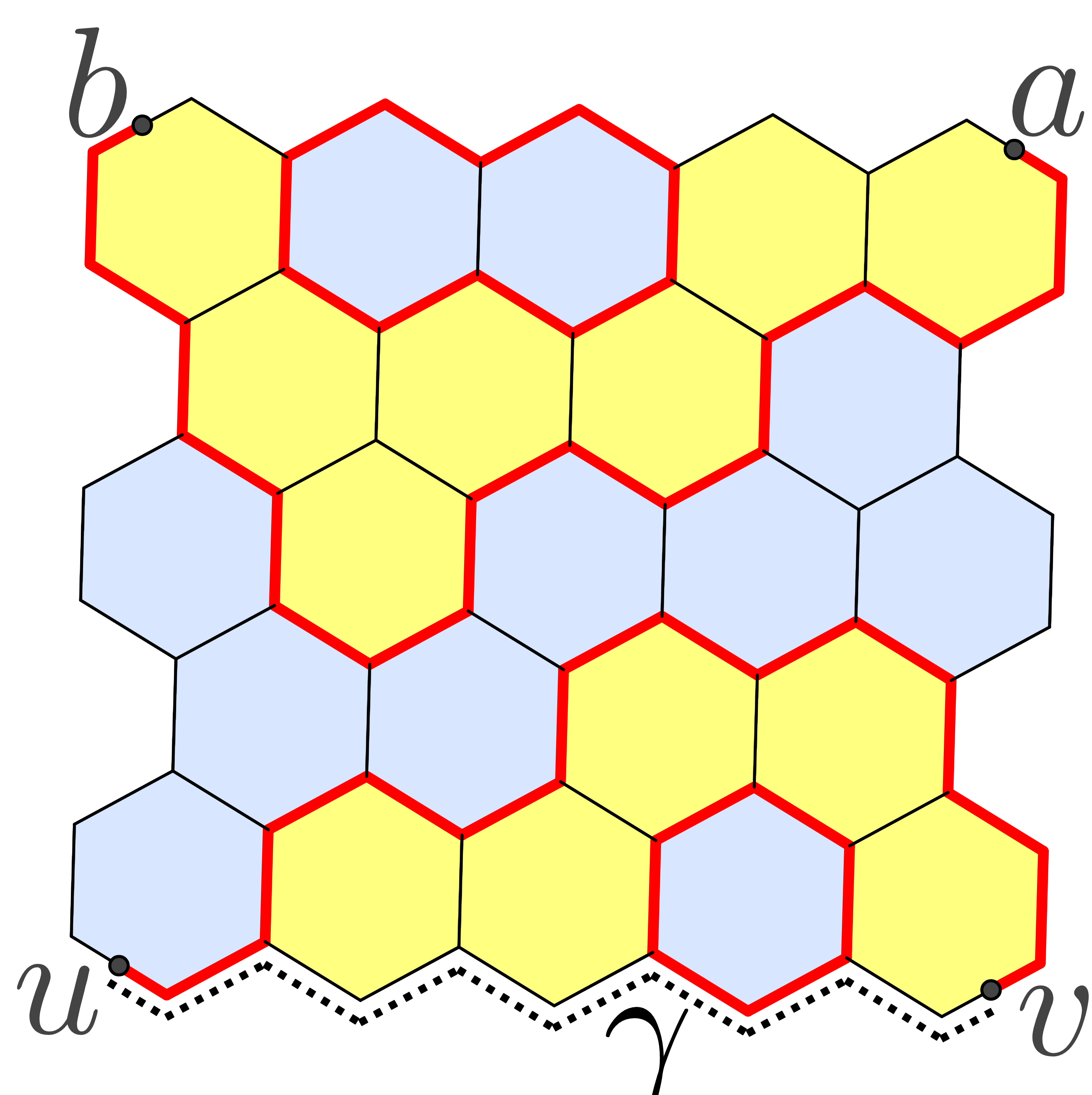}\qquad
  \includegraphics[width=0.14\textwidth]{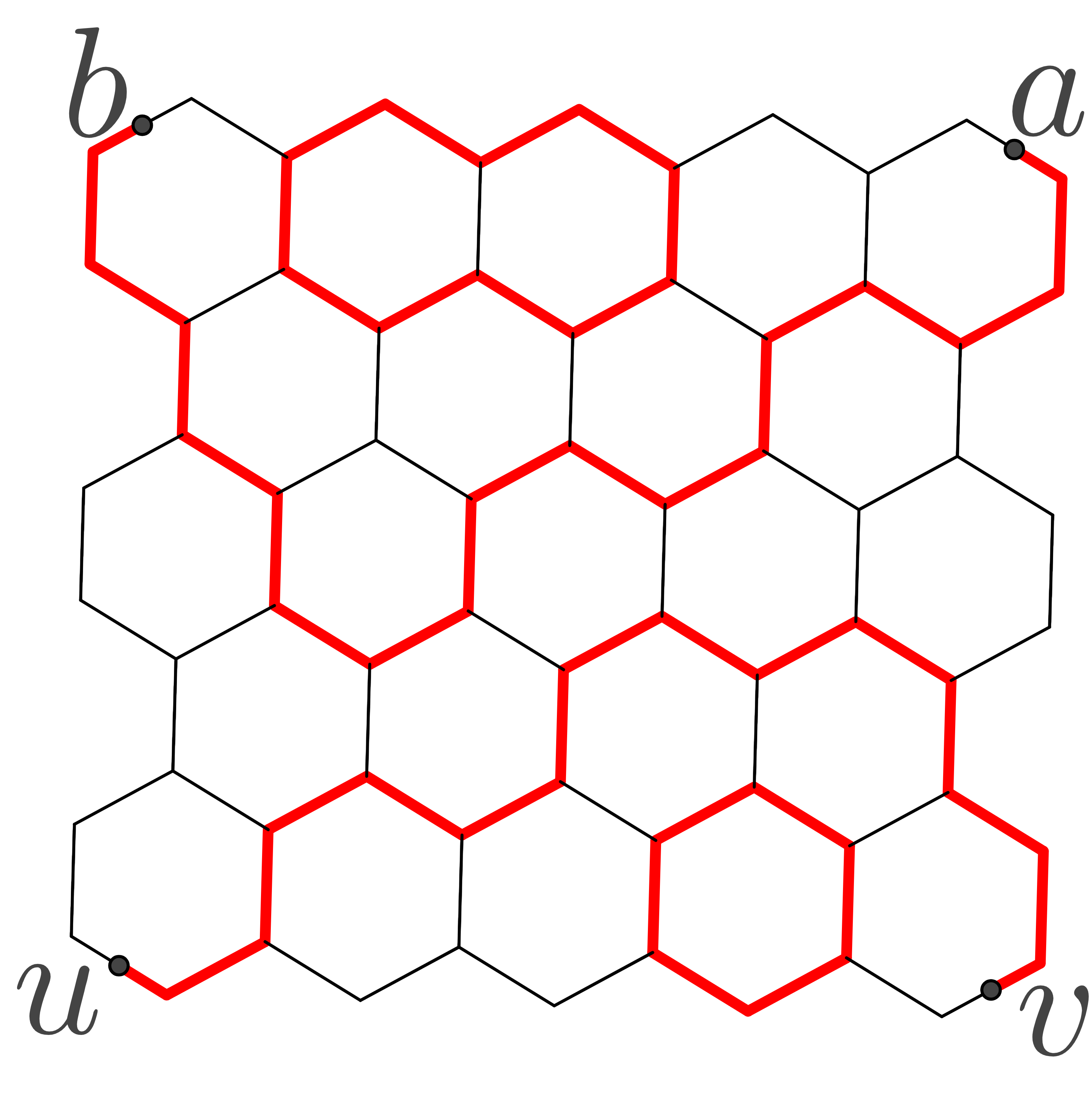}
  \caption{A coloring and a loop configuration.
    See the proof of Lemma~\ref{l-symmetric-difference}.}
   \label{fig:Dobrushin}
\end{figure}

In what follows a \emph{hexagon} is a hexagon of the lattice approximation $\Omega^\delta$. A \emph{midpoint} is a midpoint of a side of a hexagon. A \emph{half-side} is a segment joining a midpoint with an endpoint of the same side. A \emph{broken line} is a simple broken line (possibly closed) consisting of half-sides, viewed as a subset of the plane. A \emph{broken line} $pq$ is such a broken line with distinct endpoints $p$ and $q$. A \emph{loop configuration with disorders at $a,b,u,v$} (or just \emph{loop configuration} for brevity) is a disjoint  union of several broken lines, with exactly two being non-closed and having the endpoints at $a,b,u,v$ and all the other ones being closed. It is easy to see (see, for instance, \cite[\S1.2]{KS-20})  that the number of loop configurations equals the number of colorings of hexagons in two colors. See Figure~\ref{fig:Dobrushin}.

\begin{figure}[htbp]
  \centering
      \includegraphics[width=0.7\textwidth]{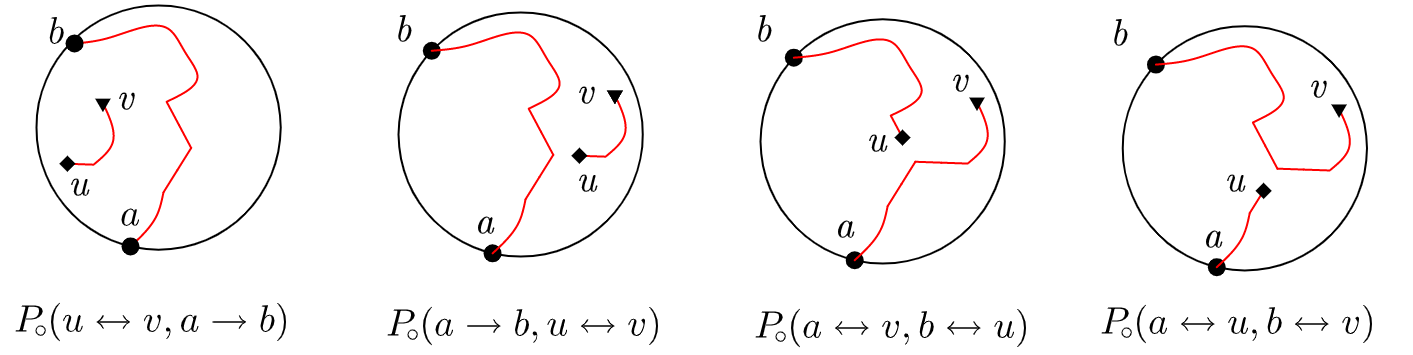}
      \caption{Link patterns and the probabilities they contribute to}
      \label{fig:patterns}
\end{figure}

Denote by $\ploop(a\leftrightarrow u,b\leftrightarrow v)$ the fraction of loop configurations containing a broken line $au$ (and hence another broken line $bv$ as well). Define $\ploop(a\leftrightarrow v,b\leftrightarrow u)$ analogously. See Figure~\ref{fig:patterns}.

Now consider loop configurations containing a broken line $ab$.
The broken line (which we always orient from $a$ to $b$) divides the polygon $\Omega^\delta$ into connected components, each bordering upon $ab$ either from the right or from the left. We say that those connected components \emph{lie to the right} and \emph{to the left} from $ab$ respectively.
Denote by $\ploop(a\rightarrow b,u\leftrightarrow v)$ the fraction of loop configurations containing broken lines $ab$ and $uv$ such that $uv$ lies to the right from $ab$. Define $\ploop(u\leftrightarrow v,a\rightarrow b)$ analogously.
Beware the different meaning of notations $\ploop(a\leftrightarrow v,b\leftrightarrow u)$ and $\ploop(u\leftrightarrow v,a\rightarrow b)$.

We have decomposed the set of all loop configurations into $4$ subsets (called \emph{link patterns}) depending on the arrangement of broken lines with the endpoints $a,b,u,v$ (see Figure~\ref{fig:patterns}).

The \emph{parafermionic observable} is the complex-valued function on the set of pairs of midpoints, distinct from $a$ and $b$, given by the formula
\begin{multline*}
  \hspace{-0.9cm}
  F(u,v):=
  \begin{cases}
   \ploop(a\leftrightarrow v,b\leftrightarrow u)
  -\ploop(a\leftrightarrow u,b\leftrightarrow v)
  +i\sqrt{3}\ploop(u\leftrightarrow v,a\rightarrow b) -i\sqrt{3}\ploop(a\rightarrow b,u\leftrightarrow v),
  & \mbox{if }u\ne v, \\
   i\sqrt{3}P(u\cdots u, a\rightarrow b)
  -i\sqrt{3}P(a\rightarrow b,u\cdots u),
  & \mbox{if }u=v.
  \end{cases}
\end{multline*}

The following lemma and theorem, being the main result in full strength, will imply Theorem~\ref{th-main}.

\begin{lemma}\label{l-symmetric-difference}
For any distinct midpoints $a,b,u,v$ we have $P(u\cdots v, a\rightarrow b)=\ploop(u\leftrightarrow v,a\rightarrow b)$ and  $P(a\rightarrow b,u\cdots v)=\ploop(a\rightarrow b,u\leftrightarrow v)$. Hence $P(u\cdots v, a\rightarrow b)-P(a\rightarrow b,u\cdots v)=\mathrm{Im}\,F(u,v)/\sqrt{3}$.
\end{lemma}

The latter automatically holds for $u=v$ as well.

\begin{theorem}[Continuum limit of the parafermionic observable] \label{prop-convergence}
Let $(\Omega^\delta,a^\delta,b^\delta,u^\delta,v^\delta)$ be a lattice approximation of a domain $(\Omega,a,b,u,v)$ with two marked distinct boundary points $a,b\in\partial \Omega$ and two other points $u,v\in \overline{\Omega}$ such that $\partial\Omega$ is a closed smooth curve. Then
  \begin{equation}\label{eq-limit-function}
  F(u^\delta,v^\delta)\to 
  g\left(
\frac{\psi(u)+\overline{\psi(v)}}{\psi(u)-\overline{\psi(v)}}\right)
\qquad\text{as }\delta\searrow 0.
  \end{equation}
\end{theorem}

The proof of this theorem uses the following properties of the parafermionic observable.

\begin{lemma}[Conjugate antisymmetry] \label{l-conjugate-symmetry}
  The function $F$ is {\emph{conjugate-antisymmetric}}, i.e., $$F(u,v)=-\overline{F(v,u)}.$$
\end{lemma}

\begin{lemma}[Discrete analyticity]\label{l-discrete-analiticity}
  Let $z$ be a common vertex of $3$ hexagons.
  Let $p,q,r$ be the midpoints of their common sides in the counterclockwise order.
  Then for each midpoint $v\ne p,q,r,a,b$ we have
  \begin{equation}\label{eq-l-discrete-analiticity}
    (p-z)F(p,v)+(q-z)F(q,v)+(r-z)F(r,v)=0.
  \end{equation}
\end{lemma}

\begin{corollary}[Cauchy's formula]\label{cor-Cauchy}
Let $\gamma=w_0w_1\dots w_{n-1}$ with $w_n:=w_0$ be a closed broken line with the vertices at the centers of hexagons such that the hexagons centered at $w_j$ and $w_{j+1}$ share a side for each $j=0,\dots,n-1$. Denote by {$p_j$} the midpoint of the side $w_jw_{j+1}$. Assume that $a,b,v$ lie outside $\gamma$. Then the \emph{discrete integral of $F$ along the contour $\gamma$} defined by the formula
$$ \int\limits_{\gamma}^\# F(z,v)  \, d^{\#}z  := \sum\limits_{j=0}^{n-1} F({p_j},v) (w_{j+1}-w_{j})
$$
vanishes.
\end{corollary}

Denote by $[z;w]$ the straight-line segment with the endpoints $z,w\in\mathbb{C}$.

\begin{lemma}[Boundary values] \label{l-boundary-values}
Take distinct midpoints $a,b,u,v\in \partial \Omega^\delta$.
Let us go around $\partial \Omega^\delta$ counterclockwise and write the order of these $4$ points. Then
\begin{equation}\label{eq-boundary-values}
    F(u,v)\in
    \begin{cases}
      [+1; +i\sqrt{3}], & \mbox{if the order is } a,b,u,v;\\
      [-1; +i\sqrt{3}], & \mbox{if the order is } a,b,v,u;\\
      [-1; -i\sqrt{3}], & \mbox{if the order is } a,u,v,b;\\
      [+1; -i\sqrt{3}], & \mbox{if the order is } a,v,u,b;\\
      [-1; +1],         & \mbox{if the order is } a,u,b,v
                                     \mbox{ or } a,v,b,u.
    \end{cases}
\end{equation}
\end{lemma}

By the \emph{inradius} of a closed bounded domain $(\overline{\Omega},a,b,c)$ with three marked points we mean the minimal radius of a disk $D$ such that no two points among $a,b,c$ belong to the same connected component of $\overline{\Omega}-D$. For instance, if $\overline{\Omega}$ is the triangle with the vertices $a,b,c$ then the inradius of $(\overline{\Omega},a,b,c)$ equals the usual inradius of the triangle, i.e. the radius of the inscribed circle. We are going to apply the following two lemmas to $c=v$.

\begin{lemma}[Uniform H\"olderness] \label{l-uniform-holderness}
There exist $\eta,C>0$ such that for any distinct midpoints $a,b,c,u,w$ of a lattice approximation $\Omega^\delta$ of an {arbitrary bounded simply-connected domain and any broken line~$uw$} we have
  \begin{equation*}
    |F(u,c)-F(w,c)|\le C\left(\frac{\mathrm{diam}\, uw}{R}\right)^\eta,
  \end{equation*}
where $R$ is the inradius of $(\Omega^\delta,a,b,c)$.
The same inequality remains true, if we allow $u,w\in\{a,b,c\}$ and
set $F(a,c):=-1$, $F(b,c):=+1$.
\end{lemma}

\begin{remark*}
{Here $\mathrm{diam}\,uw$ cannot be replaced by~$|u-w|$ in general, e.g., for $\Omega=\mathrm{Int}\,D^2-[0;1]$.}
\end{remark*}

\begin{lemma}[Geometry of a lattice approximation] \label{l-lattice-approximation}
Let $(\Omega^\delta,a^\delta,b^\delta,c^\delta,u^\delta,w^\delta)$ be a lattice approximation of the domain $(\Omega,a,b,c,u,w)$ with three distinct marked points $a,b,c\in\overline{\Omega}$ and two more marked points $u,w\in\overline{\Omega}$ such that
$\partial\Omega$ is a closed smooth curve. Then there exist $C_\Omega,R_{\Omega,a,b,c}>0$ not depending on $u,w,\delta$ such that:
\begin{enumerate}
  \item $|u-u^\delta|<C_\Omega \delta$;
  \item $u^\delta$ and $w^\delta$ can be joined by a broken line $u^\delta w^\delta$ of diameter less than $C_\Omega(|u-w|+\delta)$;
  \item the inradius of $(\Omega^\delta,a^\delta,b^\delta,c^\delta)$ is greater than $R_{\Omega,a,b,c}$.
\end{enumerate}
\end{lemma}

\begin{remark*}
The ratio $|u-u^\delta|/\delta$ can be arbitrarily large for a triangle $\overline{\Omega}$ with a small angle at its vertex~$u$. Using this observation, one can construct a Jordan domain such that $|u-u^\delta|/\delta$ is unbounded.
\end{remark*}

\begin{lemma}[{Solution of the resulting boundary-value problem}] \label{l-mapping} Let $\Omega$ be a domain bounded by a closed smooth curve. Then the right-hand side of~\eqref{eq-limit-function} (continuously extended to all $u,v\in \overline{\Omega}$ except for $v=a,b$) is the unique function in $\overline{\Omega}\times (\overline{\Omega}-\{a,b\})$ that is:
\begin{enumerate}
  \item conjugate-antisymmetric in $(\overline{\Omega}-\{a,b\})\times (\overline{\Omega}-\{a,b\})$;
  \item analytic in $\Omega\times \{v\}$ and continuous in $\overline{\Omega}\times \{v\}$ for each $v\in \overline{\Omega}-\{a,b\}$; and
  \item satisfies boundary conditions~\eqref{eq-boundary-values} (with $a,b,u,v$ understood as points of $\partial \Omega$ rather than $\partial \Omega^\delta$).
\end{enumerate}
\end{lemma}

\begin{remark*} The lemma is easily generalized to an arbitrary bounded simply-connected domain, if $\partial \Omega$ is understood as the set of prime ends.
\end{remark*}

All the assertions stated above are proved in the next section, except for Proposition~\ref{l-schwarz} and the latter three technical lemmas proved in the appendix. The proof of Lemma~\ref{l-uniform-holderness} is completely analogous to \cite[Lemma~10]{KS-20}; the other proofs in the appendix are obtained by well-known methods.

\section{Proofs} \label{sec-proofs}

\begin{proof}[Proof of Lemma~\ref{l-symmetric-difference}]
Fix distinct midpoints $a$, $b$, $u$, $v$. Let us construct a bijection between the loop configurations containing broken lines $uv$ and $ab$, the former to the left from the latter, and the colorings of hexagons such that $u$ and $v$ can be joined by a broken line $\gamma$ lying to the left from the interface. For that purpose, for each possible interface of such colorings, fix such a broken line $\gamma$, so that $\gamma$ depends only on the interface but not a particular coloring.

Take a loop configuration containing broken lines $uv$, $ab$ and let us construct a coloring with the interface $ab$. Perform the symmetric difference of the loop configuration and the fixed broken line $\gamma$ (and take the topological closure). We get a disjoint union of closed broken lines and just one non-closed broken line $ab$. The desired coloring is then determined by the following two conditions:
\begin{enumerate}
  \item  two hexagons with a common side have different colors if and only if the side is contained in the resulting union;
  \item the hexagon containing the midpoint $a$ is blue, if and only if the half-side of $ab$ starting at $a$ goes counterclockwise around the boundary of the hexagon.
\end{enumerate}
Clearly, this gives the desired bijection. See Figure~\ref{fig:Dobrushin}.

A similar bijection shows that the total number of loop configurations equals the total number of colorings \cite[\S1.2]{KS-20}. Thus $P(u\cdots v, a\rightarrow b)=\ploop(u\leftrightarrow v,a\rightarrow b)$. Analogously, $P(a\rightarrow b,u\cdots v)=\ploop(a\rightarrow b,u\leftrightarrow v)$.
\end{proof}

\begin{proof}[Proof of Lemma~\ref{l-conjugate-symmetry}]
This is a direct consequence of the definition because $\ploop(u\leftrightarrow v,a\rightarrow b)=\ploop(v\leftrightarrow u,a\rightarrow b)$ and $\ploop(a\rightarrow b,u\leftrightarrow v)=\ploop(a\rightarrow b,v\leftrightarrow u)$.
\end{proof}

\begin{proof}[Proof of Lemma~\ref{l-discrete-analiticity}]
(Cf.~\cite[Proof of Lemma~4]{KS-20})
Rewrite~\eqref{eq-l-discrete-analiticity} in the form
\begin{equation}\label{eq-p-l-discrete-analiticity}
  \sum_{u\in\{p,q,r\}}(u-z)\left[
   \#(a\leftrightarrow v,b\leftrightarrow u)
  -\#(a\leftrightarrow u,b\leftrightarrow v)  +i\sqrt{3}\#(u\leftrightarrow v,a\rightarrow b) -i\sqrt{3}\#(a\rightarrow b,u\leftrightarrow v)
  \right]=0,
\end{equation}
where $\#(a\leftrightarrow u,b\leftrightarrow v)$ denotes the number of loop configurations containing broken lines $au$ and $bv$ etc.

We group loop configurations with $u=p,q,r$ respectively in triples such that any two loop configurations in a same triple differ by two half-sides adjacent to $z$. See Figure~\ref{fig:mano}. Each triple contributes zero to the left-hand side of~\eqref{eq-p-l-discrete-analiticity}. Indeed, if the loop configurations in a triple belong to the same link pattern (as in Figure~\ref{fig:mano} to the top and to the middle), then the contributions of the loop configurations to~\eqref{eq-p-l-discrete-analiticity} are proportional to $p-z$, $q-z$, $r-z$, hence sum up to zero.
If the loop configurations in a triple do not belong to the same link pattern (as in Figure~\ref{fig:mano} to the bottom), then  up to a cyclic permutation of $p,q,r$ and a permutation of $a,b$, they contain the broken lines $aq$, $br$, and $vp$ respectively, the latter lying to the left from $ab$. Up to overall minus sign, such loop configurations contribute  $(p-z)i\sqrt{3}$, $-(q-z)$, $(r-z)$ respectively to~\eqref{eq-p-l-discrete-analiticity}, again summing up to zero. This proves~\eqref{eq-p-l-discrete-analiticity}, and hence~\eqref{eq-l-discrete-analiticity}.
\end{proof}

\begin{figure}[ht!]
    \centering
    \includegraphics[width=0.80\textwidth]{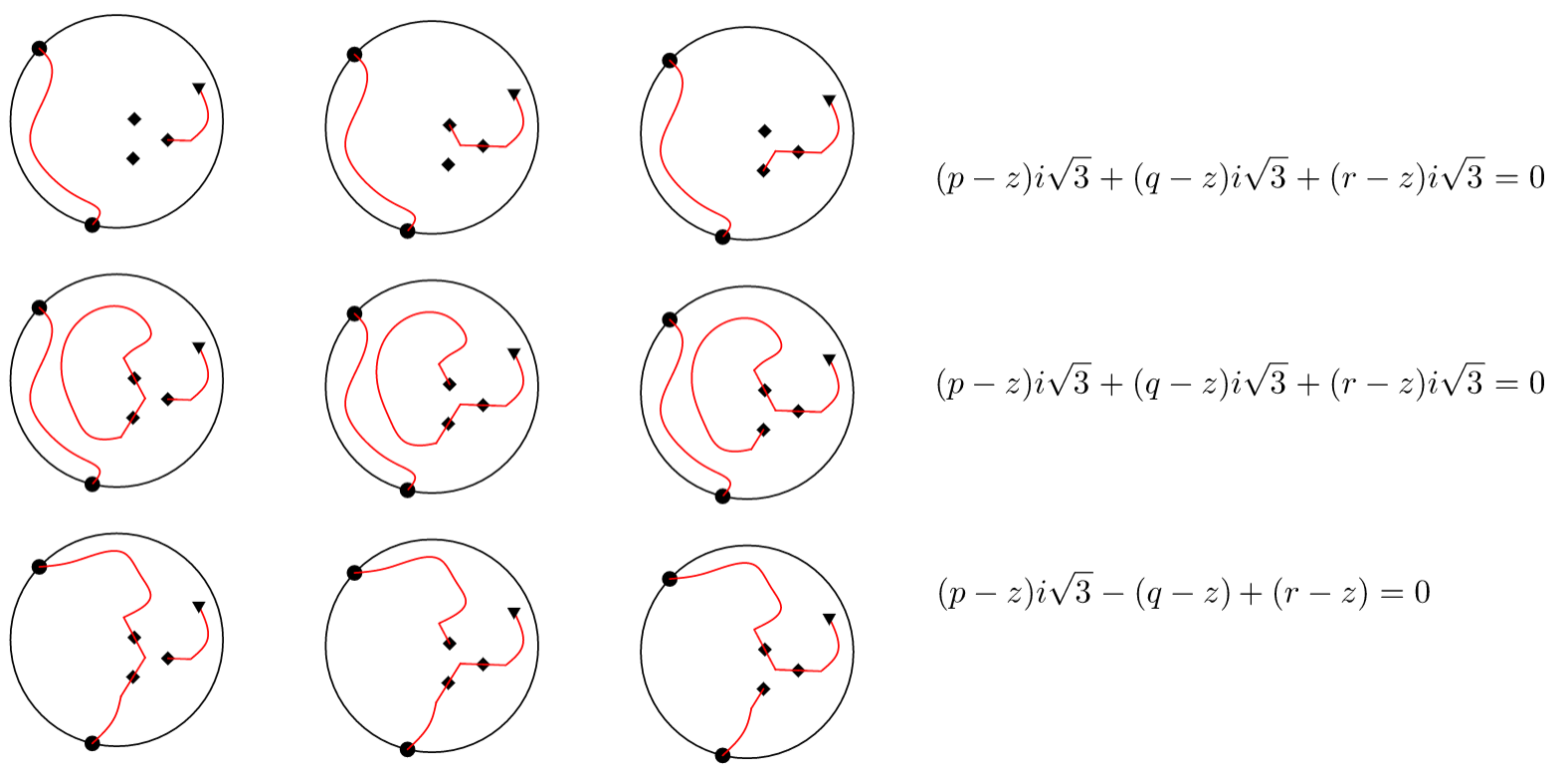}
    \caption{(Cf.~\cite[Fig.~2]{KS-20}) Graphical proof of Lemma~\ref{l-discrete-analiticity}.
    Midpoints  $p,q,r$ are marked with diamonds, $a,b$ with circles, and $v$ with triangles. Configurations are grouped horizontally.}
    \label{fig:mano}
\end{figure}

\begin{proof}[Proof of Corollary~\ref{cor-Cauchy}]
For a counterclockwise triangular contour $\gamma=w_0w_1w_2$, the corollary
follows from~\eqref{eq-l-discrete-analiticity} because $w_{j+1}-w_j= 2\sqrt{3}i({p_j}-z)$, where $z$ is the vertex enclosed by the contour.
Since any closed broken line can be tiled by triangular contours and the discrete integration is additive with respect to contour, the corollary follows.
\end{proof}

\begin{proof}[Proof of Lemma~\ref{l-boundary-values}]
Since each loop configuration belongs to exactly one of the four link patterns, it follows that
\begin{equation*}
  \ploop(a\leftrightarrow u,b\leftrightarrow v)+\ploop(a\leftrightarrow v,b\leftrightarrow u) +\ploop(u\leftrightarrow v,a\rightarrow b) +\ploop(a\rightarrow b,u\leftrightarrow v)=1.
\end{equation*}
If the order of the points is $a,b,u,v$, then no loop configurations containing disjoint broken lines $au$ and $bv$ are possible. Neither loop configurations containing a broken line $uv$ to the right from $ab$ are possible. Hence $\ploop(a\leftrightarrow u,b\leftrightarrow v)=\ploop(a\rightarrow b,u\leftrightarrow v)=0$ in this case and thus
\begin{equation*}
F(u,v)=\ploop(a\leftrightarrow v,b\leftrightarrow u)
+i\sqrt{3}(1-\ploop(a\leftrightarrow v,b\leftrightarrow u))\in [1;i\sqrt{3}].
\end{equation*}
The other orders are considered analogously.
\end{proof}

\begin{proof}[Proof of Theorem~\ref{prop-convergence}]
Step 1: piecewise linear extension of $F$. Fix $a^\delta,b^\delta,v^\delta$ and consider the function $F^\delta(u^\delta):=F(u^\delta,v^\delta)$ on the set of all midpoints $u^\delta\ne a^\delta,b^\delta$. Set $F^\delta(a^\delta):=-1$ and $F^\delta(b^\delta):=+1$. Extend the function to the centers, vertices and side midpoints of all the hexagons intersecting $\Omega$ (not just contained in $\Omega^\delta$) by the formula $F^\delta(u):=F^\delta(u^\delta)$, i.e., set
the value at a given point $u$ to be the same as at an arbitrary closest midpoint $u^\delta$. Then extend the function linearly to each triangle spanned by adjacent vertex, side midpoint, and center of a hexagon. Finally, restrict the function to $\overline{\Omega}$. We get a continuous piecewise-linear function $F^{\delta}\colon \overline{\Omega}\to\mathbb{C}$.

Step 2: extraction of a converging subsequence ${F^{\delta_n}}$. By the definition, $|F^\delta(u)|\le 2+2\sqrt{3}$, hence $F^\delta$ is uniformly bounded.

Since $\partial\Omega$ is smooth, by Lemmas~\ref{l-uniform-holderness}--\ref{l-lattice-approximation}
$F^\delta$ is uniformly H\"older in $\overline{\Omega}$. Indeed, the lemmas imply
$$
|F^\delta(u)-F^\delta(w)|\le C\left(\frac{C_\Omega(|u-w|+4\delta+\delta)}{R_{\Omega,a,b,v}}\right)^\eta
$$
for all $u,w\in \overline{\Omega}$. This gives the H\"older condition for $|u-w|\ge\delta$. Applying this inequality for a triangle spanned by adjacent vertex, side midpoint, and hexagon center, we get $|\nabla F^\delta| \le C_{\Omega,a,b,v}\delta^{\eta-1}$ inside the triangle for some $C_{\Omega,a,b,v}$ not depending on $\delta$. This implies the H\"older condition for $|u-w|<\delta$.

Then by the Arzel\`a--Ascoli theorem, there is a continuous function $f \colon  \overline{\Omega}\to \mathbb C$ and a subsequence  $\delta_n\searrow 0$ such that ${F^{\delta_n}}\rightrightarrows f$ uniformly in $\overline{\Omega}$. Hence,
\begin{equation}\label{eq-uniform-convergence}
F^{\delta_n}(u^{\delta_n})\rightrightarrows f(u)\qquad\text{as }\quad {n\to\infty}.
\end{equation}

Step 3: analyticity of the limit $f$. Take an arbitrary triangular contour $\gamma\subset \Omega$ such that $v$ is outside of~$\gamma$. Let $\gamma^\delta$ be the closed broken line with the vertices at the centers of the hexagons of $\Omega^\delta$ of maximal enclosed area contained inside $\gamma$. Then $v^\delta$ is outside $\gamma^\delta$ for sufficiently small $\delta$. Approximating an integral by a sum, applying~\eqref{eq-uniform-convergence} and Corollary~\ref{cor-Cauchy}, we get
$$\int\limits_\gamma f(z) \, dz =
\int\limits_{{\gamma^{\delta_n}}}^\# f(z) \, d^{\#}z  + o(1) =
\int\limits_{{\gamma^{\delta_n}}}^\# {F^{\delta_n}}(z) \, d^{\#}z  + o(1) =
o(1)\qquad\text{as }\quad {n\to\infty}.
$$
Thus $\int_\gamma f(z) \, dz = 0$, and by Morera's theorem $f$ is analytic in $\Omega-\{v\}$. Since $f$ is continuous in the whole $\Omega$, by the removable singularity theorem it follows that $f$ is analytic in the whole $\Omega$.

Step 4: boundary values of the limit $f$. Let us show that the function $f\colon  \overline\Omega\to \mathbb C$ satisfies boundary conditions~\eqref{eq-boundary-values} (with $a,b,u,v$ understood as points of $\partial \Omega$ rather than midpoints). Indeed, if the order of the points on $\partial \Omega$ is, say, $a,b,u,v$, then the order of $a^\delta,b^\delta,u^\delta,v^\delta$ on $\partial \Omega^\delta$ is the same for sufficiently small $\delta$. By Lemma~\ref{l-boundary-values} we have $F(u^{\delta_n},v^{\delta_n})\in [1;i\sqrt{3}]$. By~\eqref{eq-uniform-convergence},
$f(u)\in [1;i\sqrt{3}]$ as well.

Step 5: identification of the limit $f$. Recall that the function $f(u)$ depends on the parameter $v$ as well, and write $f(u,v):=f(u)$. By Lemma~\ref{l-conjugate-symmetry} and~\eqref{eq-uniform-convergence} it follows that the function $f(u,v)$ is conjugate-antisymmetric for $u,v\notin \{a,b\}$. Then by Lemma~\ref{l-mapping} the function $f(u,v)$ coincides with the right-hand side of~\eqref{eq-limit-function}. Thus the limit function $f(u)$ is uniquely determined by $a,b,v$, and thus does not depend on the choice of the converging subsequence $F^{\delta_n}$. Hence convergence~\eqref{eq-uniform-convergence} holds for the initial sequence $F^\delta$, not just a subsequence. We have arrived at~\eqref{eq-limit-function}.
\end{proof}

\begin{remark*}
Step~2 is the only one where the smoothness of the boundary is essentially used. For more general domains, this step is more technical; see \cite[{\S5}]{KS-20}.
\end{remark*}

\begin{proof}[Proof of Theorem~\ref{th-main}]
This follows directly from Lemma~\ref{l-symmetric-difference} and Theorem~\ref{prop-convergence}.
\end{proof}

Let us show that Cardy's and Schramm's formulae are indeed particular cases of this result.

\begin{proof}[Proof of Corollary~\ref{cor-Cardy}]
Apply Theorem~\ref{th-main}. Since the counterclockwise order of the marked points on $\partial \Omega$ is $a,b,v,u$, it follows that $\psi(u)/\overline{\psi(v)}\in [0;1]$ and $P(a^\delta \rightarrow b^\delta,u^\delta\cdots v^\delta)=0$ for sufficiently small $\delta$. Clearly, $P(a^\delta u^\delta\leftrightarrow b^\delta v^\delta)=P(u^\delta \cdots v^\delta,a^\delta \rightarrow b^\delta)$. Hence the crossing probability tends to the right-hand side of~\eqref{eq-th-main} as $\delta\searrow 0$.

It remains to prove that the right-hand sides of~\eqref{eq-th-main} and~\eqref{eq-cor-Cardy} are equal for $\psi(u)/\overline{\psi(v)}\in[0;1]$, i.e. $\frac {3\Gamma(2/3)}{\Gamma(1/3)^2} \eta^{1/3}\cdot
{}_2F_1 \left(\frac 13, \frac 23; \frac 43; \eta \right)=\frac{1}{\sqrt{3}}\mathrm{Im}\,g\left(\frac{\eta+1}{\eta-1}\right)$
for each $\eta\in[0;1]$.
Here the left-hand side is the Schwarz triangle function that conformally maps the upper half-plane $\mathrm{Im}\,\eta>0$ onto the equilateral triangle with the vertices $0,1,(1+\sqrt{3} i)/2$ and takes  $0,1,\infty$ to the respective vertices \cite[Ch.~VI, \S5]{Nehari-75}. By the definition and a symmetry argument, $\frac{1}{2}g\left(\frac{\eta+1}{\eta-1}\right)+\frac{1}{2}$ is the conformal mapping of the lower half-plane $\mathrm{Im}\,\eta<0$ onto the same triangle, with the images of $1$ and $\infty$ interchanged.
Since a conformal mapping onto a domain is determined by the images of three boundary points, it follows that the images of each $\eta\in[0;1]$ under the two conformal mappings are symmetric with respect to the bisector of the angle with the vertex $0$. Since for each $z\in [0;1]$ the imaginary part of the symmetric point equals $\frac{\sqrt{3}}{2}z$, it follows that the right-hand sides of~\eqref{eq-th-main} and~\eqref{eq-cor-Cardy} are equal.
\end{proof}

\begin{remark*} Corollary~\ref{cor-Cardy} can also be (more directly) deduced from Theorem~\ref{prop-convergence} applied for the case when the counterclockwise order of the marked points on $\partial\Omega$ is $a,u,b,v$.
\end{remark*}

\begin{proof}[Proof of Corollary~\ref{cor-Schramm}]
Apply Theorem~\ref{th-main} for $\Omega=D^2$, $u=v=0$, and $u^\delta=v^\delta$. We have
$$
P(0,a^\delta \rightarrow b^\delta)-P(u^\delta \cdots u^\delta,a^\delta \rightarrow b^\delta)\to 0\quad\text{ and }\quad
P(u^\delta \cdots u^\delta,a^\delta \rightarrow b^\delta)+P(a^\delta \rightarrow b^\delta,u^\delta\cdots u^\delta)\to 1
$$
as $\delta\searrow 0$ because the probability that the interface $a^\delta b^\delta$ intersects the $\delta$-neighborhood of the origin tends to $0$. (The latter well-known fact is actually reproved in the proof of Lemma~\ref{l-uniform-holderness} in the appendix.)
Thus by~\eqref{eq-th-main} we get
$$
P(0,a^\delta \rightarrow b^\delta)\to \frac{1}{2}+\frac{1}{2\sqrt{3}}
\mathrm{Im}\,g\left(
\frac{\psi(0)+\overline{\psi(0)}}{\psi(0)-\overline{\psi(0)}}\right)
\qquad\text{as }\delta\searrow 0.
$$
Here
$$
\frac{\psi(0)+\overline{\psi(0)}}{\psi(0)-\overline{\psi(0)}}
= \frac{a+b}{a-b}=-i\cot\frac{\theta}{2},
$$
because the value $a/b$ is the cross-ratio of the points $a,b,0,\infty$, the value $\psi(0)/\overline{\psi(0)}=\psi(0)/\psi(\infty)$ is the cross-ratio of their $\psi$-images, and the linear-fractional mapping $\psi\colon\mathrm{Int}\,D^2\to\{z:\mathrm{Im}z>0\}$ extended to $\mathbb{C}\cup\{\infty\}$ preserves cross-ratios. Together with Proposition~\ref{l-schwarz}, this gives~\eqref{eq-cor-Schramm}.
\end{proof}

\subsection*{Acknowledgements}

The work is supported by the Swiss NSF, ERC Advanced Grants 340340 and 741487,
National Center of Competence in Research (NCCR SwissMAP), and Theoretical Physics and Mathematics Advancements Foundation ``Basis'' grant 21-7-2-19-1. Section~3 was written entirely under the support of Russian Science Foundation grant 19-71-30002. Appendix~A was written entirely under the support of the Ministry of Science and Higher Education of the Russian Federation (agreement no. 075-15-2022-287).

The authors are grateful to I.~Benjamini, O.~Feldheim, G.~Kozma, I.~Novikov for useful discussions.

\appendix

\section{Proofs of technical lemmas} \label{sec-technical}

\begin{proof}[Proof of Lemma~\ref{l-uniform-holderness}]
The lemma follows from the sequence of inequalities to be explained below:
\begin{multline*}
\frac{|F(u,c)-F(w,c)|}{2\sqrt{3}}
\le  \ploop(a,b,c\leftrightarrow uw)\\
\le  \min\{\ploop(a,c\leftrightarrow uw), \ploop(b,c\leftrightarrow uw), \ploop(a,b\leftrightarrow uw)\}
\le  P(\partial B_R^\delta(u)\leftrightarrow \partial B_r^\delta(u))
\le  C_{\mathrm{RSW}}\left(\frac{r}{R}\right)^\eta.
\end{multline*}
Here we use the following notation:
\begin{itemize}
\item $\ploop(a,b,c\leftrightarrow uw)$ is the fraction of loop configurations with disorders at $a,b,c,u$ such that both non-closed broken lines have common points with $uw$. Those loop configurations are called \emph{tripod} in what follows.
\item $\ploop(a,c\leftrightarrow uw)$ is the fraction of loop configurations with disorders at $a,c$ such that the only non-closed broken line has common points with $uw$. Those loop configurations are called \emph{bipod} in what follows. {\emph{Loop configurations with disorders at $a,c$} are defined analogously to the ones with disorders at $a,b,c,u$.}
\item $B_R^\delta(u)$ and $B_r^\delta(u)$ are the lattice approximations of the disks centered at $u$ of radii $R$ and $r:=3\mathrm{diam}\, uw$ respectively (the factor of $3$ guarantees that $B_r^\delta(u)\ne\emptyset$ for $u\ne w$).
\item $P(\partial B_R^\delta(u)\leftrightarrow \partial B_r^\delta(u))$ is the probability that there is a \emph{crossing between $\partial B_R^\delta(u)$ and $\partial B_r^\delta(u)$} in a coloring of $\Omega^\delta$, i.e. a connected component of the union of the hexagons of the same color having common points with both $\partial B_R^\delta(u)$ and $\partial B_r^\delta(u)$.
\item $C_{\mathrm{RSW}}$ and $\eta$ are the constants  from the Russo--Seymour--Welsh inequality \cite[Proposition~8]{KS-20}.
\end{itemize}
We assume that $R-10\delta>r>0$, otherwise there is nothing to prove.

Step 1: the first inequality. The symmetric difference with the broken line $uw$ does not change the link pattern of a loop configuration unless both non-closed broken lines in the configuration have common points with $uw$. Thus each pair of non-tripod loop configurations with the symmetric difference $uw$ contributes zero to $F(u,c)-F(w,c)$. Each tripod loop configuration contributes at most $2\sqrt{3}$ times the probability of the configuration, which gives the first inequality.

Step 2: the second inequality. Let us prove that, say, $\ploop(a,b,c\leftrightarrow uw)
\le  \ploop(a,c\leftrightarrow uw)$. In what follows we are not going to use that $a,b\in\partial \Omega^\delta$, hence the same argument will prove that $\ploop(a,b,c\leftrightarrow uw)
\le  \ploop(b,c\leftrightarrow uw)$ and $\ploop(a,b,c\leftrightarrow uw)
\le  \ploop(a,b\leftrightarrow uw)$.

It suffices to construct an injection from the set of tripod loop configurations to
bipod ones.

Take a tripod loop configuration. Denote by $ax, by, cz$ the connected parts of the broken lines from $a,b,c$ respectively to the first intersection point with $uw$. Fix a broken line $bu$ (not necessary from the loop configuration) having no common points with neither $ax$ nor $cz$ (possibly except the endpoints $x$ and $z$). Such a broken line exists, e.g., the union of $by$ and the part of $uw$ from $y$ to $u$ will work. But we fix $bu$ so that it depends only on $ax$ and $cz$ but not on the tripod loop configuration itself. Then the symmetric difference with $bu$ gives the required injection, implying the second inequality.

Step 3: the third inequality. Choose two points from $a,b,c$, say, $a$ and $c$, belonging to the same connected component of $\Omega^\delta-\mathrm{Int}\,B_R^\delta(u)$. Such a pair of points exists by the definition of the inradius $R$. Join $a$ and $c$ by a fixed broken line $ac$ outside $\mathrm{Int}\,B_R^\delta(u)$.

It suffices to construct an injection from the set of bipod loop configurations to the set of colorings with a crossing between $\partial B_R^\delta(u)$ and $\partial B_r^\delta(u)$.

Take a bipod loop configuration. Perform its symmetric difference with the fixed broken line $ac$ (and take the closure). We get a disjoint union of closed broken lines. The resulting union determines a coloring of the hexagons (by condition~1 from the proof of Lemma~\ref{l-symmetric-difference} and the requirement that the hexagon containing $a$ is painted blue), which we assign to the initial loop configuration.

At least one of the resulting closed broken lines has common points with both $ac$ and $uw$, hence with both $\partial B_R^\delta(u)$ and $\partial B_r^\delta(u)$, otherwise the symmetric difference of $ac$ and the union of all the resulting closed broken lines would not connect $a$ with $uw$. The hexagons bordering upon that closed broken line from inside have the same color. Thus the constructed coloring has a crossing, which implies the third inequality.

Step 4: the fourth inequality. Notice that removing the hexagons outside the ring between $\partial B_R^\delta(u)$ and $\partial B_r^\delta(u)$ does not affect the probability of a crossing, and adding hexagons inside the ring (not contained in $\Omega^\delta$ initially) can only increase this probability. This way the third inequality reduces to the Russo--Seymour--Welsh inequality for a ring \cite[Proposition~8]{KS-20}.

This concludes the proof {for $u,w\not\in\{a,b,c\}$. Otherwise, the argument is analogous, only some loop configurations contain fewer disorders.}
\end{proof}

\begin{proof}[Proof of Lemma~\ref{l-lattice-approximation}]
For an open square $\Pi$ with side length $s$ and a real number $\varepsilon <s$ denote by $\Pi^{-\varepsilon}$ the square with side length $s-\varepsilon$ having the same center and side directions.
By the implicit function theorem and a compactness argument,
we can choose $\varepsilon>0 $ and open squares $\Pi_1,\dots,\Pi_m$ such that
\begin{enumerate}
\item $\overline{\Omega}$ is covered by  $\cup_{j=1}^m \Pi_j^{-100 \varepsilon}$  (in particular, each $\Pi_j$ has side length at least $100 \varepsilon$);
\item for each $j$ either $\Pi_j\subset\Omega$ or
  the  intersection $\overline{\Omega} \cap \Pi_j$  has a form $\{(x, y) \in\Pi_j : y \ge h(x)\}$ for some function $h(x)$ with $h(0) = 0$ and $|h'(x)| < 1/100$,
  where the coordinate axes pass through the center of $\Pi_j$ and are parallel to the sides.
\end{enumerate}

Take $\delta< \varepsilon$. Then each hexagon of the lattice contained in $\Omega$ is contained in some $\Pi_j$ by condition~1, all the hexagons in $\Omega \cap \Pi_j$ form a connected polygon by condition~2, and the union of such polygons for all $j$ is connected. So, the union of hexagons of the lattice contained in $\Omega$  is connected and hence  coincides with $\Omega^\delta$. Set  $C_\Omega:= \mathrm{diam}\, \Omega / \varepsilon > 100$ and let us prove properties 1--3 stated in the lemma.

Property 1. If $\delta\ge\varepsilon$ then there is nothing to prove. Otherwise take any $u \in\Omega$. Then $u \in\Omega\cap\Pi_j^{-100 \varepsilon}$ for some $j$.
The distance between $u$ and $\Omega^\delta$ is at most $2\delta$ (because
by condition~2 the point $u+(0,2\delta)$ belongs to a hexagon contained in $\Omega\cap\Pi_j$, hence to $\Omega^\delta$). Thus $| u^\delta - u| < 3\delta < C_\Omega \delta$.

Property 2. If $|u-w|\ge \varepsilon$ or $\delta \ge \varepsilon$ then there is nothing to prove. Otherwise $| u^\delta - u|,| w^\delta - w| < 3\delta$ and $u, w, u^\delta, w^\delta\subset \Omega\cap\Pi_j$ for some $j$. It suffices to join
$u^\delta$ and $w^\delta$ by a broken line $u^\delta w^\delta$ of length at most $6|u^\delta-w^\delta|$. If $[u^\delta;w^\delta]$ does not intersect $\partial\Omega^\delta$, then there is a broken line $u^\delta w^\delta$ of length at most $2|u^\delta-w^\delta|$. If $[u^\delta;w^\delta]$ does intersect $\partial\Omega^\delta$, then consider the points $u':=u^\delta+(0,|u^\delta-w^\delta|)$ and $w':=w^\delta+(0,|u^\delta-w^\delta|)$. Then the three segments
$u^\delta u'$, $u'w'$, and $w'w^\delta$ of total length $3|u^\delta-w^\delta|$ are contained in $\Omega^\delta$. Hence $u^\delta$ and $w^\delta$ can be joined by a broken line $u^\delta w^\delta$ of length at most $6|u^\delta-w^\delta|$.

Property 3. Join $a,b,{c}$ by curves $ab,bc,ca\subset\overline{\Omega}$ having no common points besides the endpoints. The domain $(abc,a,b,c)$ enclosed by the curves has a positive inradius $R$. Decompose each curve into arcs of diameter at most $\delta$. For each arc $uw$, the points $u^\delta$ and $w^\delta$ are $C_\Omega\delta$-close to $u$ and $w$ (property~1) and can be joined by a broken line of diameter at most $2C_\Omega\delta$ (property~2). Thus $a^\delta,b^\delta,c^\delta$ can be joined by three broken lines (possibly having common points), $3C_\Omega\delta$-close to $ab,bc,ca$. Hence the inradius of $(\Omega^\delta,a^\delta,b^\delta,c^\delta)$ is greater than $R/2$ for $\delta<R/6C_\Omega$ (and greater than $\delta/4$ for all~$\delta$).
\end{proof}

\begin{proof}[Proof of Lemma~\ref{l-mapping}]
\emph{Existence.} Let us prove the right-hand side of~\eqref{eq-limit-function} has all the properties listed in the lemma. Denote $z:=\frac{\psi(u)+\overline{\psi(v)}}{\psi(u)-\overline{\psi(v)}}$ so that the right-hand side equals $g(z)$.

Property 1: conjugate-antisymmetry. We have $g(z)=-g(-z)=\overline{g(\overline{z})}=-\overline{g(-\overline{z})}$ for all $z$,
because the four conformal mappings have the same domain, the same image, and the same images of $0,\pm 1$. Since interchanging $u$ and $v$ takes $z$ to $-\overline{z}$, the right-hand side of~\eqref{eq-limit-function} is conjugate-antisymmetric.

Property 2: analyticity and continuity. For fixed $v\ne a,b$, the right-hand side of~\eqref{eq-limit-function} is analytic in $\Omega$ and continuous in $\overline{\Omega}$ because
$z\in(-\infty;-1]\cup[1;+\infty)$ only if both
$u,v$ belong to the counterclockwise arc $ab$ of $\partial\Omega$ or both $u,v$ belong to the clockwise arc $ab$ (see the convention right after Theorem~\ref{th-main}).

Property 3: boundary values. If the order of the points on $\partial\Omega$ is $a,b,u,v$, then $z\in[1;+\infty)$. Then by the convention right after Theorem~\ref{th-main}, the value $g(z)$ is understood as continuously extended from the first coordinate quadrant. Since $g(z)=\overline{g(\overline{z})}=-\overline{g(-\overline{z})}$ for all $z$,
it follows that $g$ maps the first quadrant to the triangle with the vertices $0,1,i\sqrt{3}$, maps $[0;1]$ onto itself, and takes the positive imaginary semi-axis $[0;+i\infty]$ to $[0;i\sqrt{3}]$. By the boundary correspondence principle, $g(z)\in [1;i\sqrt{3}]$.
The other orders of $a,b,u,v$ are considered analogously. We arrive at~\eqref{eq-boundary-values}.

\emph{Uniqueness.} Let us prove that each function $f(u,v)$ satisfying the conditions of the lemma, coincides with the right-hand side of~\eqref{eq-limit-function}. Consider the following $4$ cases:

Case 1: $v\in\partial\Omega-\{a,b\}$. For fixed $v$, the function $f(u,v)$ is analytic in $\Omega$, continuous in $\overline{\Omega}$, and takes the arcs $ab$, $bv$, $va$ of $\partial\Omega$ to the segments $[-1;1]$, $[1;\pm i\sqrt{3}]$, $[\pm i\sqrt{3};-1]$ respectively (where the sign depends on the clockwise order of $a$, $b$, $v$ along $\partial\Omega$). By the argument principle, it follows that such function $f(u,v)$ is unique, hence it coincides with the right-hand side of~\eqref{eq-limit-function} for $v\in\partial\Omega-\{a,b\}$.

Case 2: $v\in\Omega$ and $u\in\partial\Omega-\{a,b\}$. By the conjugate antisymmetry, we have $f(u,v)=-\overline{f(v,u)}$. The latter coincides with the right-hand side of~\eqref{eq-limit-function} by Case 1 and by conjugate antisymmetry of the right-hand side of~\eqref{eq-limit-function}.

Case 3: $v\in\Omega$ and $u\in\{a,b\}$. By the continuity, the values $f(u,v)$ for $u\in\partial\Omega-\{a,b\}$ determined in Case 2 uniquely determine the values $f(a,v)$ and $f(b,v)$.

Case 4: $v,u\in\Omega$. For fixed $v$, the function $f(u,v)$ is analytic in $\Omega$ and continuous in $\overline{\Omega}$. Thus the boundary values (determined in Cases 2--3) uniquely determine the function $f(u,v)$. Thus $f(u,v)$ coincides with the right-hand side of~\eqref{eq-limit-function} for all $(u,v)\in \overline{\Omega}\times (\overline{\Omega}-\{a,b\})$.
\end{proof}

\begin{proof}[Proof of Proposition~\ref{l-schwarz}]
We use well-known properties of the Schwarz triangle functions~\cite[Ch.~VI,\S5]{Nehari-75}.

The Schwarz triangle function
$g_0(z):=\frac {3\Gamma(2/3)}{\Gamma(1/3)^2}
z^{1/3}\cdot {}_2F_1 \left(\frac 13, \frac 23, \frac 43, z \right)$
maps the upper half-plane onto the equilateral triangle with the vertices $0,1,(1+\sqrt{3}i)/2$, and takes $0,1,\infty$ to the respective vertices. The composition $2 g_0((z+1)/2)-1$ maps the half-plane
onto the triangle with the vertices $-1,1,i\sqrt{3}$, and takes $-1,1,\infty$ to the respective vertices. It also takes $0$ to $0$ by symmetry because conformal mapping is uniquely determined by the images of three points. The Schwarz reflection principle shows that the same composition maps $\mathbb{C}-[1;+\infty)\cup(-\infty;-1]$ onto the interior of the rhombus.

The Schwarz triangle function $\frac{2\sqrt{3}\Gamma(2/3)}{\sqrt{\pi}\Gamma(1/6)} z^{1/2}\cdot{}_2F_{1}\left(\frac{1}{2},\frac{2}{3};\frac{3}{2};
z \right)$ maps the upper half-plane onto the right triangle with the vertices $0,1,i\sqrt{3}$, and takes $0,1,\infty$ to the respective vertices. The precomposition with the mapping $z\mapsto z^2$ takes the first quadrant $\{z:\mathrm{Re}\,z,\mathrm{Im}\,z>0\}$ to the same triangle. Applying the Schwarz reflection principle, we see that the same composition maps $\mathbb{C}-[1;+\infty)\cup(-\infty;-1]$ onto the interior of the rhombus, and $0,+1,-1$ are fixed.
\end{proof}

\noindent
\textsc{Mikhail Khristoforov\\
Saint Petersburg University}
\\
\texttt{mikhail.khristoforov\,@\,gmail$\cdot $com}

\noindent
\textsc{Mikhail Skopenkov\\
King Abdullah University of Science and Technology \&\\
HSE University (Faculty of Mathematics) \&\\
Institute for Information Transmission Problems of the Russian Academy of Sciences} 
\\
\texttt{mikhail.skopenkov\,@\,gmail$\cdot $com} \quad \url{https://users.mccme.ru/mskopenkov/}

\noindent
\textsc{Stanislav Smirnov\\
Universit\'e de Gen\`eve, Gen\`eve 4, Switzerland \&\\
Saint Petersburg University
{\&\\
Skolkovo Institute of Science and Technology
}
}
\\
\texttt{Stanislav.Smirnov@unige.ch}

\end{document}